\documentclass[12pt,a4paper]{article}
\usepackage[utf8]{inputenc}
\usepackage[T1]{fontenc}
\usepackage[total={17cm,25cm}, top=2cm, left=2cm, includefoot]{geometry}
\usepackage{amsmath}
\usepackage{amssymb}
\usepackage{amsthm}
\usepackage{xcolor}
\usepackage{subcaption}
\usepackage{graphicx}
\usepackage{hyperref}
\usepackage[numbers]{natbib}

\theoremstyle{definition}
\newtheorem{definition}{Definition}
\newtheorem{conjecture}[definition]{Conjecture}
\newtheorem{theorem}[definition]{Theorem}
\newtheorem{lemma}[definition]{Lemma}
\newtheorem{proposition}[definition]{Proposition}

\newtheorem{corollary}[definition]{Corollary}

\newcommand{\dist}{\operatorname{dist}}

\title{Cages and cyclic connectivity}
\author{Robert Lukoťka, Edita Máčajová, Jozef Rajník\\[2mm]
\normalsize Department of Computer Science\\
\normalsize Comenius University, 842 48 Mlynská Dolina, Bratislava, Slovakia\\
\normalsize\tt \{luktoka, macajova, rajnik\}@dcs.fmph.uniba.sk
}

\begin{document}

\maketitle

\begin{abstract}
A graph $G$ is \emph{cyclically $c$-edge-connected} if there is no set of fewer than $c$ edges that disconnects $G$ into at least two cyclic components. We prove that if a $(k, g)$-cage $G$ has at most $2M(k, g) - g^2$ vertices, where $M(k, g)$ is the Moore bound, then $G$ is cyclically $(k - 2)g$-edge-connected, 
which equals the number of edges separating a $g$-cycle,
and every cycle-separating $(k - 2)g$-edge-cut in $G$ separates a cycle of length $g$. In particular, this is true for unknown cages with $(k, g) \in \{(3, 13), (3, 14), (3, 15), (4, 9), (4, 10)$, $(4, 11),$ $(5, 7), (5, 9), (5, 10), (5, 11), (6, 7), (9, 7)\}$ and also the potential missing Moore graph with degree $57$ and diameter $2$.

\medskip
\noindent\textbf{Keywords:} cage, cyclic connectivity, girth, lower bound
\end{abstract}

\section{Introduction}

A \emph{$(k, g)$-graph} is a $k$-regular graph with girth $g$. The well-known cage problem, also known as the degree-girth problem, asks to find a~smallest $(k, g)$-graph, called a \emph{$(k, g)$-cage}, whose order is denoted by $n(k, g)$. The cage problem was first introduced by Tutte \cite{Tutte47} in 1947 and since then it has been widely studied \cite{Exoo}.

In general, finding the order of cages is difficult. A universal lower bound on the size of cages, also called the \emph{Moore bound}, says that $n(k, g) \ge M(k, g)$, where
$$M(k, 2d + 1) = 1 + k + k(k - 1) + \dots + k(k - 1)^{d - 1} = \frac{k(k-1)^d - 2}{k - 2}$$
and
$$M(k, 2d + 2) = 2 + 2(k - 1) + 2(k - 1)^2 + \dots + 2(k- 1)^{d} = \frac{2(k - 1)^{d + 1} - 2}{k - 2}.$$
This comes from the observation that the set of all edges of a $(k, g)$-graph at distance at most $d = \lfloor(g - 3)/2\rfloor$ from a fixed vertex for odd $g$, or from a fixed edge for even $g$, induces a tree where all non-leaf vertices have degree $k$ and all leaves are in depth $d$.

If a $(k, g)$-cage attains this bound, it is called a \emph{Moore graph}, or precisely \emph{$(k, g)$-Moore graph}. For odd $g$, Moore graphs can be equivalently characterised as $k$-regular graphs with diameter $(g - 1)/2$ and order $M(k, g)$, which is the largest possible order under these requirements. Moore graphs are almost completely characterised. The most famous open case is a~hypothetical $(57, 5)$-Moore graph (having diameter $2$ and order $3250$), whose existence has not been settled.

Except for these cases, only a small finite number of cages is known.
Most of them have been found by various exhaustive searches, usually with the assistance of a~computer.
So far, the best general upper bound on $n(k, g)$ is roughly $M(k, g)^2$ \cite{Sauer}. Also, there are no significant improvements to the Moore bound.
For instance, Jajcayová, Filipovski and Jajcay \cite{Jajcayova16} showed that $n(k,g) > M(k, g) + 4$ for some infinite class of parameters $k$ and $g$.
Therefore, the gap between the lower and upper bound on $n(k, g)$ is rather large. For a detailed survey on the cage problem, we refer the reader to the dynamic survey \cite{Exoo}. 

Since finding the exact values of $n(k, g)$ is difficult,
various properties of cages are studied, connectivity properties among them. Fu et al. \cite{Fu} conjectured that all simple $(k, g)$-cages are $k$-connected. Balbuena and Salas \cite{Balbuena12} showed that each $(k, g)$-cage is $(\lfloor k/2 \rfloor + 1)$-connected for every odd $g \ge 7$ and Lin et al. \cite{Lin08} showed that they are $(r + 1)$-connected for each even $g$ and $r^3 + 2r^2 \le k$. In contrast to the vertex-connectivity, the edge-connectivity of cages is known to be $k$, proven by Wang et al. for odd girth \cite{Wang03} and by Lin et al. for even girth \cite{Lin05}.

However, (edge-)connectivity itself does not faithfully describe the structural properties of $k$-regular graphs, especially for small $k$, since it is bounded by $k$ from above.
Therefore, various modifications of (edge-)connectivity are considered throughout the literature that forbid ``trivial'' cuts, usually by requiring some additional properties of the components obtained by the removal of a cut.
Marcote et al. \cite{Marcote04} together with Lin et al. \cite{Lin06} showed that all $(k, g)$-cages are \emph{edge-superconnected}, that is, they are $k$-edge-connected and each $k$-cut consists of $k$ edges incident with the same vertex. In other words, there are forbidden components containing only one vertex.
For a survey on various properties required from the components, we refer the reader to \cite{Harary83}.

According to Tutte \cite{Tutte60}, we consider only edge-cuts whose removal yields at least two components that contain a cycle. Such edge-cuts are called \emph{cycle-separating}. A graph $G$ is \emph{cyclically $k$-edge-connected} if it contains no cycle-separating edge-cut consisting of fewer than $k$ edges.
Note that a graph with minimum degree $k$ is edge-superconnected if and only if it is cyclically $(k + 1)$-connected.
Cyclic edge-connectivity has emerged as an important invariant of cubic graphs. It appears in approaches to numerous widely-open problems \cite{Glover, Kochol04_5f_cc6, Macajova20_bf_cc5, Robertson17, Zhang}. Several results \cite{Doslic03, Lovasz86, McCuaig92} and conjectures \cite{Jaeger} suggest that high cyclic edge-connectivity ensures various structural properties of cubic graphs. 
Despite its importance, cyclic edge-connectivity has not been studied in the area of cages.
Therefore, we shall investigate it in this paper.

First, we illustrate a natural lower bound. Consider a $(k, g)$-cage $G$ for $k \ge 3$ and $g \ge 5$, and let $C_g$ be a $g$-cycle of $G$. We can separate $C_g$ by the removal of $(k - 2)g$ edges. For $k \ge 3$ and $g \ge 5$, the subgraph $G - C_g$ has at least $n(k, g) - g \ge g$ vertices, so it contains a cycle. Hence, the $(k, g)$-cage $G$ is at most cyclically  $(k - 2)g$-edge-connected.
As it happens for the usual edge-connectivity, it is reasonable to conjecture that the cyclic edge-connectivity of cages also attains this maximal possible value. Thus, we propose the following conjecture.

\begin{conjecture}
    \label{conj:cc}
    Each $(k, g)$-cage is cyclically $(k - 2)g$-edge-connected.
\end{conjecture}

Moreover, we introduce the following strengthening saying that each cycle-separating of the conjectured maximal size separates a $g$-cycle.

\begin{conjecture}
    \label{conj:g-cut}
    For each $(k, g)$-cage $G$, any cycle-separating $(k - 2)g$-edge-cut in $G$ separates a~$g$-cycle.
\end{conjecture}

In this paper, we show that Conjectures \ref{conj:cc} and \ref{conj:g-cut} hold for all $(k, g)$-cages with order at most $2M(k, g) - g^2$. This covers all Moore graphs, including the hypothetical $(57, 5)$-Moore graph. 
Thus, if the hypothetical Moore graph with degree $57$ and girth $5$ exists, its cyclic edge-connectivity attains the maximal possible value of $55\cdot 5 = 275$ and each cycle-separating $275$-edge-cut separates a $5$-cycle.
Using the known upper bounds on $n(k, g)$, we also conclude that
Conjectures \ref{conj:cc} and \ref{conj:g-cut} are true for some small values of $k$ and $g$, in particular, for unknown cages with $(k, g) \in \{(3, 13), (3, 14), (3, 15), (4, 9), (4, 10)$, $(4, 11),$ $(5, 7), (5, 9), (5, 10), (5, 11), (6, 7),\linebreak (9, 7)\}$

The idea of our approach is based on the observation that a component obtained after the removal of an edge-cut from some $k$-regular graph can be viewed as a $k$-regular ``graph'' with semiedges allowed, called a multipole.
Therefore, we extend the cage problem to multipoles with $s$ semiedges, introducing the number $n(k, g, s)$ denoting the order of a smallest ``nontrivial'' $(k, g, s)$-multipole.
Besides its connection to cages, the problem of estimating the value $n(k, g, s)$ is also interesting on its own, for instance, in the study of subgraphs of graphs with high girth that can be separated by small edge-cuts.
Also, a similar problem has been studied by Nedela and Škoviera \cite{Nedela95}, who estimated the size of atoms of cyclic edge-connectivity.

In Proposition \ref{prop:moore-bound-multipoles}, we provide a Moore-like lower bound on the order $n$ of a~$(k, g, s)$-multipole.
Based on this, we show in Propositions \ref{prop:quadratic-odd} and \ref{prop:quadratic-even} that $n$  satisfies some quadratic inequality and in Proposition \ref{prop:simple-bounds}, we estimate that except for several small values of $k$ and $g$ we have either $n \le g^2/2$ or $n \ge M(k, g) - g^2/2$.
Our main result is captured in Theorem \ref{thm:kgs-pole-lower}, where we prove that the first of the two inequalities holds only in cases where the multipole is acyclic or is a single cycle of length $g$.
Therefore, if a $(k, g)$-graph $G$ has less than $2M(k, g) - g^2$ vertices, then it is cyclically $(k - 2)g$-connected and each cycle-separating $(k - 2)g$-edge cut in $G$ separates a $g$-cycle. So both Conjectures \ref{conj:cc} and \ref{conj:g-cut} hold.

\section{Multipoles with prescribed girth}
\label{sec:multipoles}

As we mentioned in Introduction, we generalise the cage problem to multipoles -- graphs with \emph{semiedges}, which are edges incident with only one vertex. We start by explaining some terminology.
An edge whose ends are incident with two distinct vertices is called a \emph{link}.
The \emph{order} $|G|$ of the multipole $G$ is the number of its vertices. The \emph{distance} $\dist_G(u, v)$ between the vertices $u$ and $v$ in a multipole $G$ is the number of edges in a shortest $u$-$v$-path. If there is no $u$-$v$-path, we set $\dist_G(u, v) = \infty$. For $X, Y \subseteq V(G)$, we define $\dist_G(X, Y) = \min\{\dist_G(x, y) \mid x \in X, y \in Y\}$. Here, $X$ or $Y$ can be an edge of $G$ which is treated as a set of vertices incident with it. The \emph{girth} of a multipole $G$ is the length of its shortest cycle or $\infty$ if $G$ is acyclic.
The \emph{degree} of a vertex $v \in V(G)$ is the number of edges (including semiedges) incident with $v$. If we want to exclude the semiedges, we use the term \emph{inner degree} of $v$.
A multipole is \emph{$k$-regular} if each its vertex has degree $k$. Since a graph is a special case of a multipole, all of these definitions can also be applied to graphs.

A \emph{$(k, g, s)$-multipole} is a $k$-regular multipole with girth \emph{at least} $g$ and exactly $s$ semiedges.
Observe that the notion of a $(k, g, s)$-multipole includes various trivial instances. Each acyclic $k$-regular multipole with $s$ semiedges is a $(k, g, s)$-multipole for each $g$ and the cycle of length $\ell \ge g$ with each vertex incident with $k - 2$ semiedges is a $(k, g, \ell(k - 2))$-multipole for any $k \ge 3$.
However, if a~$(k, g)$-cage $G$ would contain a cycle-separating edge-cut $S$ violating Conjecture \ref{conj:cc} or \ref{conj:g-cut}, then the components of $G - S$ are definitely different from the trivial instances mentioned above.
Therefore, for the study of Conjectures 1 and 2, we exclude these multipoles by the following definition.

A $(k, g, s)$-multipole with $s \le (k - 2)g$ is called \emph{nontrivial} if it is cyclic and different from a $g$-cycle. Otherwise, it is called \emph{trivial}. We denote the order of a smallest nontrivial $(k, g, s)$-multipole by $n(k, g, s)$.
Note that if $s < (k - 2)g$, the condition of being cyclic is sufficient to ensure nontriviality, since a $(k, g, s)$-multipole cannot be a $g$-cycle in this case.
Motivated by Conjectures \ref{conj:cc} and \ref{conj:g-cut}, we focus on the cases when $s \le (k - 2)g$ and we do not extend the notion of non-triviality to
values $s > (k - 2)g$. 

For a better touch on this newly introduced number, we present the values of $n(3, g, s)$ and $n(4, g, s)$ for small $g$ and $s$ in Table \ref{tab:g3-values} and Table \ref{tab:g4-values}, respectively. Some of these multipoles are depicted in Tables \ref{tab:g3-figures} and \ref{tab:g4-figures}. We have obtained these values using a computer.

\begin{table}[h]
    \centering
\begin{tabular}{|c|c|c|c|c|c|c|c|c|c|}
\hline
    $s =$   & 0 & 1 & 2 & 3 & 4 & 5 & 6 & 7 & 8 \\\hline
    $g = 3$ & 4 & 5 & 4 & 5 &   &   &   &   &   \\\hline
    $g = 4$ & 6 & 7 & 6 & 5 & 6 &   &   &   &   \\\hline
    $g = 5$ &10 &11 &10 & 9 & 8 & 7 &   &   &   \\\hline
    $g = 6$ &14 &15 &14 &13 &12 &11 & 8 &   &   \\\hline
    $g = 7$ &24 &25 &24 &23 &22 &21 &20 &17 &   \\\hline
    $g = 8$ &30 &31 &30 &29 &28 &27 &26 &25 &22 \\\hline
\end{tabular}
\caption{The values of $n(3, g, s)$ for small $g$ and $s$}
    \label{tab:g3-values}
\end{table}

\begin{table}[h]
    \centering
\begin{tabular}{|c|c|c|c|c|c|c|c|c|c|}
\hline
    $s =$   & 0 & 2 & 4 & 6 & 8 &10 &12 &14 &16 \\\hline
    $g = 3$ & 5 & 5 & 4 & 4 &   &   &   &   &   \\\hline
    $g = 4$ & 8 & 8 & 7 & 6 & 5 &   &   &   &   \\\hline
    $g = 5$ &19 &19 &18 &17 &15 &10 &   &   &   \\\hline
    $g = 6$ &26 &26 &25 &24 &23 &22 &20 &   &   \\
\hline
\end{tabular}
\caption{The values of $n(4, g, s)$ for small $g$ and $s$}
    \label{tab:g4-values}
\end{table}

\begin{table}[h]
\newcommand{\scale}{0.7}
\begin{tabular}{ccccccccc}
$g = $  &
$s = 0$ &
$s = 1$ &
$s = 2$ &
$s = 3$ &
$s = 4$ &
$s = 5$ &
$s = 6$ &
\\
$3$ &
\includegraphics[scale=\scale]{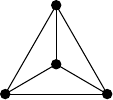} &
\includegraphics[scale=\scale]{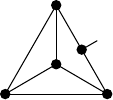} &
\includegraphics[scale=\scale]{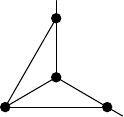} &
\includegraphics[scale=\scale]{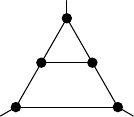} &&
\\
$4$ &
\includegraphics[scale=\scale]{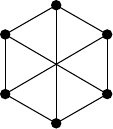} &
\includegraphics[scale=\scale]{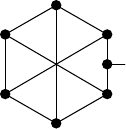} &
\includegraphics[scale=\scale]{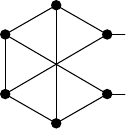} &
\includegraphics[scale=\scale]{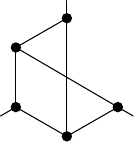} &
\includegraphics[scale=\scale]{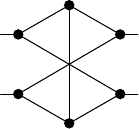} &
\\
$5$ &
\includegraphics[scale=\scale]{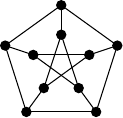} &
\includegraphics[scale=\scale]{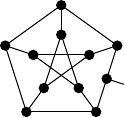} &
\includegraphics[scale=\scale]{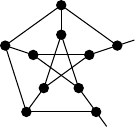} &
\includegraphics[scale=\scale]{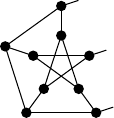} &
\includegraphics[scale=\scale]{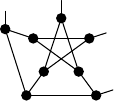} &
\includegraphics[scale=\scale]{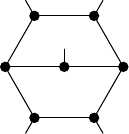} &
\\
$6$ &
\includegraphics[scale=\scale]{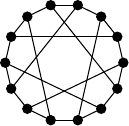} &
\includegraphics[scale=\scale]{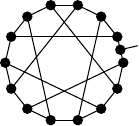} &
\includegraphics[scale=\scale]{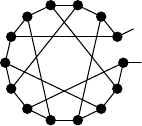} &
\includegraphics[scale=\scale]{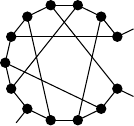} &
\includegraphics[scale=\scale]{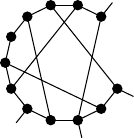} &
\includegraphics[scale=\scale]{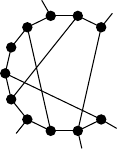} &
\includegraphics[scale=\scale]{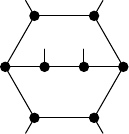} 
\\
\end{tabular}
\caption{Smallest nontrivial $(3, g, s)$-multipoles for small values of $g$ and $s$}
\label{tab:g3-figures}
\end{table}

\begin{table}[h]
\newcommand{\scale}{0.7}
\begin{tabular}{cccccccc}
$g = $ &
$s = 0$ &
$s = 2$ &
$s = 4$ &
$s = 6$ &
$s = 8$ &
$s = 10$ &
\\
$3$ &
\includegraphics[scale=\scale]{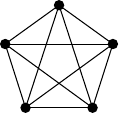} &
\includegraphics[scale=\scale]{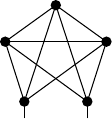} &
\includegraphics[scale=\scale]{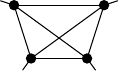} &
\includegraphics[scale=\scale]{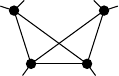} &&
\\
$4$ &
\includegraphics[scale=\scale]{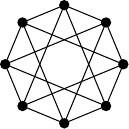} &
\includegraphics[scale=\scale]{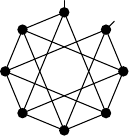} &
\includegraphics[scale=\scale]{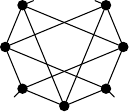} &
\includegraphics[scale=\scale]{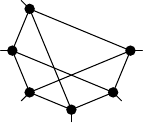} &
\includegraphics[scale=\scale]{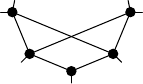} &
\\
$5$ &
\includegraphics[scale=\scale]{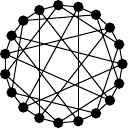} &
\includegraphics[scale=\scale]{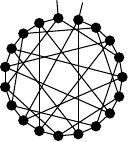} &
\includegraphics[scale=\scale]{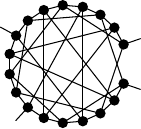} &
\includegraphics[scale=\scale]{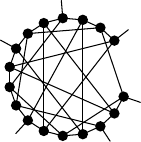} &
\includegraphics[scale=\scale]{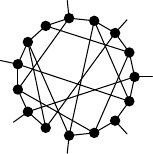} &
\includegraphics[scale=\scale]{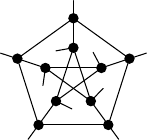}
\\
\end{tabular}
\caption{Smallest nontrivial $(4, g, s)$-multipoles for small values of $g$ and $s$}
\label{tab:g4-figures}
\end{table}

We conclude this section by proving a useful structural lemma for nontrivial $(k, g, s)$-multipoles.

\begin{lemma}
    \label{lemma:not-too-many-semiedges}
    For each $k$, $g$ and $s$, if there exists a nontrivial $(k, g, s)$-multipole $G$, then there exists a nontrivial $(k, g, s)$-multipole $H$ of the same order $|G|$ that has no vertex with inner degree $1$.
\end{lemma}

\begin{proof}
    Let $H$ be a nontrivial $(k, g, s)$-multipole of the same order as $G$ with $x$ vertices with inner degree $1$, where $x$ is the smallest possible. Suppose to the contrary that $x > 0$. Let $v$ be a vertex of $H$ incident with $k - 1$ semiedges and one link with the other end vertex $u$, and let $e$ be another link of $H$ (for instance, a link from a cycle of $H$). From $H$, we create the multipole $H'$ as follows. We remove the vertex $v$ from $H$ together with $k - 1$ semiedges incident with $v$ leaving the link $uv$ of $H$ as the semiedge in $H'$. Then, we subdivide $e$ with one new vertex $w$ incident with $k - 2$ semiedges. The multipole $H$' still has $s$ semiedges and a finite girth at least $g$. Also, $H'$ is not a $g$-cycle, for otherwise, $H$ would contain a cycle of length at most $g - 1$. Thus, $H'$ is a nontrivial $(k, g, s)$-multipole with less than $x$ vertices incident with $k - 1$ semiedges -- a contradiction. 
\end{proof}

\section{\texorpdfstring{Bounds on the size of $(k, g, s)$-multipoles}{Bounds on the size of (k, g, s)-multipoles}}
\label{sec:bounds}

In this section, we present lower bounds on the size of nontrivial $(k, g, s)$-multipoles.
In some statements, we consider only $g \ge 5$. We deal with the not-so-difficult cases $g \in \{3, 4\}$ in Section \ref{sec:results}.
Just as in the proof of the Moore bound, we shall consider the set of all the vertices of $G$ at distance $d$ from some chosen vertex or edge.
However, in contrast to $(k, g)$-graphs, this set will not form a complete Moore tree and will often contain less than $M(k, g)$ vertices depending on the depths of the semiedges.

\begin{proposition}
\label{prop:moore-bound-multipoles}
    Let $G$ be a $(k, g, s)$-multipole with $s$ semiedges $f_1, f_2, \dots, f_s$; and let $d = \lfloor (g - 1)/2\rfloor$. Let $x$ be a vertex of $G$ if $g$ is odd, and let $x$ be an edge of $G$ if $g$ is even. Then
    $$|G| \ge M(k, g) - \frac{1}{k - 2} \sum_{i = 1}^s \left((k - 1)^{d - \dist(x, f_i)} - 1\right).$$
\end{proposition}

\begin{proof}
    We consider the set $T$ of vertices at distance at most $d$ from $x$. Due to the girth of $G$, the edges at distance at most $d - 1$ from $x$ induce a tree. If $G$ contained no semiedges, then $T$ would have $M(k, g)$ vertices. Each semiedge $f_i$ removes from this Moore tree a complete $(k - 1)$-ary tree of depth $d - \dist(x, f_i) - 1$ vertices, which has $\left((k - 1)^{d - \dist(x, f_i)} - 1\right)/(k - 2)$ vertices.
    
    Note that if $\dist(x, f_i) > d$, then $-((k - 1)^{d - \dist(x, f_i)} - 1)/(k - 2) \ge - 1/(k - 1)$.
	So each semiedge $f_i$ at distance more than $d$ from $x$ increases the bound by at most $1/(k - 1)$, which is all right since the end vertex, which can be incident with at most $(k - 1)$ semiedges, is not contained in $T$.
\end{proof}

Now comes the question, which vertex or edge should we choose for counting the lower bound from Proposition \ref{prop:moore-bound-multipoles}.
It turns out that using all of them, or in other words considering the average bound through all the vertices (links), yields a sufficiently high lower bound.

\begin{proposition}
\label{prop:quadratic-odd}
Let $k \ge 3$, $d \ge 1$ and $s \ge 0$ be integers and $G$ be a $(k, 2d + 1, s)$-multipole. Then the order $n$ of $G$ satisfies the quadratic inequality
\begin{equation}
    P_{k, 2d + 1, s}(n) := n^2 - M(k, 2d + 1)n + \frac{s(dk - 2d - 1)(k - 1)^d + s}{(k - 2)^2} \ge 0.\label{eq:quadratic-odd}
\end{equation}
\end{proposition}

\begin{proof}
Let $f_1, f_2, \dots, f_s$ be the semiedges of $G$ and let $M = M(k, 2d + 1)$. According to Proposition \ref{prop:moore-bound-multipoles}, for a vertex $v \in V(G)$ we have
\begin{equation*}
    n \ge M - \frac{1}{k - 2}\sum_{i = 1}^{s}\left((k - 1)^{d - \dist(v, f_i)} - 1\right).
\end{equation*}
We sum this bound for each vertex $v$ of $G$ obtaining 
\begin{align}
    \sum_{v \in V(G)}n &\ge \sum_{v \in V(G)}M - \frac{1}{k - 2}\sum_{v \in V(G)}\sum_{i = 1}^{s}\left((k - 1)^{d - \dist(v, f_i)} - 1\right), \nonumber\\
    n^2 &\ge Mn - \frac{1}{k - 2}\sum_{i = 1}^{s}\sum_{v \in V(G)}\left((k - 1)^{d - \dist(v, f_i)} - 1\right). \label{eq:quadratic-odd-1}
\end{align}

After reversing the order of the summation, we estimate the inner sum using that there are at most $(k - 1)^h$ vertices at distance $h$ from $f_i$ and each of them contributes to the sum by the value $(k - 1)^{d - h} - 1$.

\begin{gather*}
    \sum_{v \in V(G)}\left((k - 1)^{d - \dist(v, f_i)} - 1\right) \le \sum_{h = 0}^{d - 1}(k - 1)^{h}\left((k - 1)^{d - h} - 1\right) = \\
    = \sum_{h = 0}^{d - 1}\left( (k - 1)^d - (k - 1)^h\right) = d(k-1)^d - \frac{(k - 1)^d - 1}{k - 2} = \frac{(dk - 2d - 1)(k - 1)^d + 1}{k - 2}.
\end{gather*}

We plug this back into Equation (\ref{eq:quadratic-odd-1}) which yields
\begin{align}
    n^2 &\ge Mn - \frac{1}{k - 2}\sum_{i = 1}^{s}\frac{(dk - 2d - 1)(k - 1)^d + 1}{k - 2},\nonumber\\
    n^2 &- Mn + \frac{s(dk - 2d - 1)(k - 1)^d + s}{(k - 2)^2} \ge 0.\nonumber
\end{align}
\end{proof}

\begin{proposition}
\label{prop:quadratic-even}
Let $k \ge 3$, $d \ge 1$ and $s \ge 0$ be integers and $G$ be a $(k, 2d + 2, s)$-multipole and let $M = M(k, 2d + 2)$ denote the Moore bound for a $(k, 2d + 2)$-cage. Then the order $n$ of $G$ satisfies the quadratic inequality
\begin{equation}\label{eq:quadratic-even}
P_{k, 2d + 2, s}(n) := kn^2 - (kM + s)n + Ms + 2s\cdot\frac{(dk - 2d - 1)(k - 1)^{d + 1} + k - 1}{(k - 2)^2} \ge 0.
\end{equation}
\end{proposition}

\begin{proof}
Let $f_1, f_2, \dots, f_s$ be the semiedges of $G$. According to Proposition \ref{prop:moore-bound-multipoles}, for some link $e$ of $G$ we have
\begin{equation}
    n \ge M - \frac{1}{k - 2}\sum_{i = 1}^{s}\left((k - 1)^{d - \dist(e, f_i)} - 1\right).
\end{equation}
We sum this bound for each link $e$ of $G$. Observe that $G$ has $\frac12(kn - s)$ links.
\begin{align}
    \sum_{e}n &\ge \sum_{e}M - \frac{1}{k - 2}\sum_{e}\sum_{i = 1}^{s}\left((k - 1)^{d - \dist(e, f_i)} - 1\right),\nonumber\\
    n(kn - s) &- M(kn - s) + \frac{2}{k - 2}\sum_{i = 1}^{s}\sum_{e}\left((k - 1)^{d - \dist(e, f_i)} - 1\right) \ge 0. \label{eq:quadratic-even-1}
\end{align}

Using the fact that there are at most $(k - 1)^{h + 1}$ links $e$ at distance $h$ from $f_i$, we estimate the inner sum $\sum_{e}\left((k - 1)^{d - \dist(e, f_i)} - 1\right)$ is at most

\begin{gather*}
     \sum_{h = 0}^{d - 1}(k - 1)^{h + 1}\left((k - 1)^{d - h} - 1\right)
    = \sum_{h = 0}^{d - 1}\left( (k - 1)^{d + 1} - (k - 1)^{h + 1}\right) =\\
    = d(k - 1)^{d + 1} - \frac{(k - 1)^{d + 1} - k + 1}{k - 2} = 
    \frac{(dk - 2d - 1)(k - 1)^{d + 1} + k - 1}{k - 2}.
\end{gather*}

Plugging this back to (\ref{eq:quadratic-even-1}), we have
\begin{align}
    n(kn - s) - M(kn - s) + \frac{2}{k - 2}\sum_{i = 1}^{s}\frac{(dk - 2d - 1)(k - 1)^{d + 1} + k - 1}{k - 2} &\ge 0,\nonumber\\
    kn^2 - (kM + s)n + Ms + 2s\cdot\frac{(dk - 2d - 1)(k - 1)^{d + 1} + k - 1}{(k - 2)^2} &\ge 0. \nonumber
\end{align}
\end{proof}

Propositions \ref{prop:quadratic-odd} and \ref{prop:quadratic-even} show that the order of a $(k, g, s)$-multipole $G$ satisfies the quadratic inequality $P_{k, g, s}(|G|) \ge 0$, which is either Inequality (\ref{eq:quadratic-odd}) for odd $g$, or Inequality (\ref{eq:quadratic-even}) for even $g$.
The parabola corresponding to the quadratic polynomial $P_{k,g,s}$ has its vertex at the point
$$v(k, g, s) =
\begin{cases}
\frac12 M(k, g),& \text{if } g \text{ is odd},\\
\frac12 M(k, g) + \frac{s}{2k},& \text{if } g \text{ is even}.
\end{cases}
$$
Observe that $v(k, g, s) \ge \frac12M(k, g)$ in both cases. We denote the real roots of $P_{k, g, s}$, if they exist, by $b_1(k, g, s)$ and $b_2(k, g, s)$ in such a way that $b_1(k, g, s) \le b_2(k, g, s)$. If $P_{k,g,s}$ has no real roots, we define $b_1(k, g, s) = b_2(k, g, s) = v(k, g, s)$. Thus, for each $(k, g, s)$-multipole $G$ we have
$$|G| \le b_1(k, g, s) \le v(k, g, s) \qquad \text{or} \qquad |G| \ge b_2(k, g, s) \ge v(k, g , s).$$

\begin{table}[h!]
    \centering
    \begin{tabular}{|c|ccccccc|}
\hline
g = &   $k= 3$ &   $k= 4$ &   $k= 5$ &   $k= 6$ &   $k= 7$ &   $k= 8$ &   $k= 9$ \\\hline
 3  & 1         & 2         & 3         & 3         & 3         & 3         & 3         \\
 4  & 3.3       & (5)       & (6.2)     & (7.3)     & (8.4)     & 9         & 8.2       \\
 5  & 5         & 7         & 7.2       & 7.4       & 7.7       & 7.9       & 8.0       \\
 6  & (8)       & 12        & 12.3      & 13.0      & 13.5      & 13.9      & 14.3      \\
 7  & 9.6       & 11.5      & 12.9      & 13.9      & 14.7      & 15.4      & 16.0      \\
 8  & 14        & 16.9      & 19.3      & 21.1      & 22.5      & 23.5      & 24.4      \\
 9  & 13.6      & 17.9      & 20.9      & 23.1      & 24.8      & 26.1      & 27.1      \\
10  & 18.3      & 24.6      & 29.0      & 32.2      & 34.6      & 36.4      & 37.9      \\
11  & 18.9      & 26.2      & 31.4      & 35.1      & 37.8      & 39.9      & 41.6      \\
12  & 24.3      & 34.4      & 41.3      & 46.2      & 49.8      & 52.6      & 54.7      \\
13  & 25.3      & 36.7      & 44.5      & 49.9      & 53.9      & 56.9      & 59.2      \\
14  & 31.6      & 46.4      & 56.3      & 63.1      & 68.0      & 71.8      & 74.7      \\
15  & 33.1      & 49.3      & 60.1      & 67.5      & 72.9      & 76.9      & 80.0      \\
16  & 40.3      & 60.5      & 73.7      & 82.7      & 89.2      & 94.0      & 97.8      \\
\hline
    \end{tabular}
    \caption{
    Approximate values of $b_1(k, g, (k - 2)g)$. The values in parentheses indicate the cases when 
    $b_1(k, g, (k - 2)g)=v(k, g, (k-2)g)$.}
    \label{tab:b1-bounds}
\end{table}
\begin{table}[h!]
    \centering
    \newcommand{\cell}[2]{$\footnotesize\begin{array}{c}#1 \\ #2\end{array}$}
    \begin{tabular}{|c|ccccccc|}
\hline
g = &   $k= 3$ &   $k= 4$ &   $k= 5$ &   $k= 6$ &   $k= 7$ &   $k= 8$ &   $k= 9$ \\\hline
 3  & 1         & 2         & 3         & 3         & 3         & 3         & 3         \\
 4  & 3.3       & (5)       & (6.2)     & (7.3)     & (8.4)     & 9         & 8.2       \\
 5  & 5         & 7         & 7.2       & 7.4       & 7.7       & 7.9       & 8.0       \\
 6  & (8)       & 12        & 12.3      & 13.0      & 13.5      & 13.9      & 14.3      \\
 7  & 9.6       & 11.5      & 12.9      & 13.9      & 14.7      & 15.4      & 16.0      \\
 8  & 14        & 16.9      & 19.3      & 21.1      & 22.5      & 23.5      & 24.4      \\
 9  & 13.6      & 17.9      & 20.9      & 23.1      & 24.8      & 26.1      & 27.1      \\
10  & 18.3      & 24.6      & 29.0      & 32.2      & 34.6      & 36.4      & 37.9      \\
11  & 18.9      & 26.2      & 31.4      & 35.1      & 37.8      & 39.9      & 41.6      \\
12  & 24.3      & 34.4      & 41.3      & 46.2      & 49.8      & 52.6      & 54.7      \\
13  & 25.3      & 36.7      & 44.5      & 49.9      & 53.9      & 56.9      & 59.2      \\
14  & 31.6      & 46.4      & 56.3      & 63.1      & 68.0      & 71.8      & 74.7      \\
15  & 33.1      & 49.3      & 60.1      & 67.5      & 72.9      & 76.9      & 80.0      \\
16  & 40.3      & 60.5      & 73.7      & 82.7      & 89.2      & 94.0      & 97.8      \\
\hline
    \end{tabular}
    \caption{Approximate values of $b_2(k, g, (k - 2)g)$.
    }
    \label{tab:b2-bounds}
\end{table}

\begin{lemma}
    \label{lemma:incr-decr-s-bounds}
    For any fixed $k \ge 3$ and $g \ge 3$, if the values $b_1(k, g, s)$ and $b_2(k, g, s)$ are roots of $P_{k,g,s}$, then for each $t < s$ it is true that
    $b_1(k, g, t) < b_1(k, g, s)$ and $b_2(k, g, t) > b_2(k, g, s).$
\end{lemma}

\begin{proof}
    If $g$ is odd the statement trivially follows from the fact that with decreasing~$s$, only the constant coefficient of $P_{k, g, s}$ also decreases, so the distance of its roots $b_1(k, g, s)$ and $b_2(k, g, s)$ from $v(k, g, s)$ increases.
    
    For a given even $g \ge 4$ and a given $k \ge 3$, we can express $P_{k, g, s}$ as
    $$P_{k, g, s}(n) = kn^2 - (kM + s)n + sc,$$
    where
    $$c = M + 2\cdot\frac{(dk - 2d - 1)(k - 1)^{d + 1} + k - 1}{(k - 2)^2} > M.$$
    Since $P_{k, g, s}(c) = kc^2 - kMc = kc(c - M) \ge 0$ and $c > M$, we have that $c >b_2(k, g, s)$. Thus, for each $n$, where $b_1(k, g, s) \le n \le b_2(k, g, s) < c$, we have $t(c - n) < s(c - n)$ which is equivalent to $P_{k, g, t}(n) < P_{k, g, s}(n)$. Therefore, $b_1(k, g, t) < b_1(k, g, s)$ and $b_2(k, g, s) < b_2(k, g, t)$.
\end{proof}

Lemma~\ref{lemma:incr-decr-s-bounds} shows that the values $b_1(k, g, s)$ and $b_2(k, g, s)$ are most relevant when $s = (k - 2)g$. Therefore, we list these values in Tables \ref{tab:b1-bounds} and \ref{tab:b2-bounds}.

Since the values $b_1(k, g, s)$ and $b_2(k, g, s)$, obtained as the roots of $P_{k, g, s}$, are rather complicated, it is convenient to estimate them using simpler expressions. Instead of finding a universal estimation of $b_1(k, g, s)$ and $b_2(k, g, s)$, we prefer a simpler estimation in exchange for several small exceptions where it is not true.

\begin{proposition}
    \label{prop:simple-bounds}
    Let $k$, $g$, $s$ be integers satisfying $s \le (k - 2)g$ and one of the following conditions
    \begin{itemize}
        \item $k = 3$ and $g \ge 11$,
        \item $k = 4$ and $g \ge 7$,
        \item $k \in \{5, 6\}$ and $g \ge 5$,
        \item $k \in \{7, 8, 9, 10\}$ and $g \ge 3$ and $g \ne 4$,
        \item $k \ge 11$ and $g \ge 3$.
    \end{itemize}
    Then
    $$b_1(k, g, s) \le \frac{g^2}{2} \qquad \text{and} \qquad b_2(k, g, s) \ge M(k, g) - \frac{g^2}{2}.$$
\end{proposition}

\begin{proof}
    Due to Lemma \ref{lemma:incr-decr-s-bounds}, it is sufficient to consider $s = (k - 2)g$. We show that $P_{k, g, s}(g^2/2) < 0$ which implies that $g^2/2$ is between $b_1(k, g, s)$ and $b_2(k, g, s)$.
    By the symmetry of a parabola, we obtain that $b_2(k, g, s) = 2v(k, g, s) - b_1(k, g, s) \ge M(k, g) - g^2/2$.
    Thus, it remains to show that $P_{k, g, s}(g^2/2) < 0$, which is rather technical.

    If $g = 2d + 1$, after plugging $n = g^2/2$ and $s = (k - 2)g$ in Inequality (\ref{eq:quadratic-odd}), we want to prove
    \begin{gather*}
        \frac{g^4}{4} - \frac{k(k - 1)^d - 2}{k - 2} \cdot \frac{g^2}{2} + g \cdot \frac{(kd - g)(k - 1)^d + 1}{k - 2} < 0, 
    \end{gather*}
    or equivalently,
    \begin{align}
        (k - 2)g^3 - (k(k- 1)^d - 2) \cdot (4d + 2) &+ 4(kd - g)(k - 1)^d + 4 < 0, \nonumber\\
        (k - 2)g^3 + 4g + 4 &< (4g + 2k)(k - 1)^{(g - 1)/2}.\label{eq:simple-odd}
    \end{align}

We prove (\ref{eq:simple-odd}) by splitting it into five cases, each of which requires just routine induction to prove.

\begin{itemize}
\item Case $g = 3$, $k \ge 7$: 
\begin{equation}
(k - 2)g^3 + 4g + 4 = 27k-38 < (2k+12)(k-1) = (4g + 2k)(k - 1)^{(g - 1)/2}\nonumber
\end{equation}
\item Case $g=5$, $k \ge 5$: 
\begin{equation}
(k - 2)g^3 + 4g + 4 = 125k-226 < (2k+20)(k-1)^2 = (4g + 2k)(k - 1)^{(g - 1)/2}\nonumber
\end{equation}
\item Case  $k = 3$, $g \ge 11$: 
\begin{equation}
(k - 2)g^3 + 4g + 4 < 1.25g^3 < 50 \cdot 2^{(g - 1)/2} \le (4g + 2k)(k - 1)^{(g - 1)/2}\nonumber
\end{equation}
\item Case  $k = 4$, $g \ge 7$: 
\begin{equation}
(k - 2)g^3 + 4g + 4 < 2.25g^3 < 36 \cdot 3^{(g - 1)/2} \le (4g + 2k)(k - 1)^{(g - 1)/2}\nonumber
\end{equation}
\item Case  $k \ge 5$, $g \ge 7$: 
\begin{equation}
(k - 2)g^3 + 4g + 4 < (k-1) g^3 <  (k-1) \cdot 4g \cdot 4^{(g - 3)/2} < (4g + 2k)(k - 1)^{(g - 1)/2}\nonumber
\end{equation}
\end{itemize}
   
    If $g = 2d$, according to Proposition \ref{prop:quadratic-even} we want to show that for $n = g^2/2 = 2d^2$ and $s = 2(k - 2)d$ it is true that
    \begin{equation*}
    kn^2 - kMn - sn + Ms + 2s\cdot\frac{(kd - 2d - k + 1)(k - 1)^{d}}{(k - 2)^2} + \frac{2s(k - 1)}{(k - 2)^2} < 0
    \end{equation*}
    After plugging and multiplying by $(k - 2) / (4d)$, we obtain
    \begin{gather*}
    k(k - 2)d^3 - kd(k - 1)^d + kd - (k-2)^2d^2 + (k - 2)(k - 1)^d - k + 2 + \\ + (kd - 2d - k + 1)(k - 1)^{d} + k - 1 < 0,
    \end{gather*}
which can be simplified to
    \begin{gather}
        k(k - 2)d^3 - (k - 2)^2d^2 + kd + 1  < (2d + 1)(k - 1)^d.\label{eq:simple-even}
    \end{gather}

We prove (\ref{eq:simple-even}) by splitting it into four cases, each of which requires just routine induction to prove.

\begin{itemize}
\item $d=2$, $k \ge 11$:
\begin{equation*}
k(k - 2)d^3 - (k - 2)^2d^2 + kd + 1 = 
4k^2+2k-15 < 5(k-1)^2=(2d + 1)(k - 1)^d.
\end{equation*}
\item $k \ge 5$, $d \ge 3$:
\begin{equation*}
k(k - 2)d^3 - (k - 2)^2d^2 + kd + 1  < (k - 1)^2d^3 < 
(k-1)^2(2d + 1)4^{d-2} \le (2d + 1)(k - 1)^d.
\end{equation*}
\item $k=3$, $d \ge 6$:
\begin{equation*}
3d^3 - d^2 + 3d + 1 < (2d + 1) \cdot 1.5 \cdot d^2< (2d + 1)2^d.
\end{equation*}
\item $k=4$, $d \ge 4$:
\begin{equation*}
8d^3 - 4d^2 + 4d + 1  < (2d + 1) \cdot 4 \cdot d^2 < (2d + 1)3^d.
\end{equation*}
\end{itemize}    
\end{proof}

Finally, we proceed to the main result of this section. We show that if $G$ is a~nontrivial $(k, g, s)$-multipole with $s \le (k - 2)g$, then it cannot happen that $|G| \le b_1(k, g, s)$. Thus, we obtain a nontrivial lower bound on the size of $G$.

\begin{theorem}\label{thm:kgs-pole-lower}
    Let $G$ be a nontrivial $(k, g, s)$-multipole for some $k \ge 3$, $g \ge 5$ and $s \le (k - 2)g$. Then
    $$|G| \ge b_2(k, g, s) \ge \frac12 M(k, g).$$
\end{theorem}

\begin{proof}
    Note that $b_2(k, g, s) \ge \frac12 M(k, g)$.
    It is easy to check that the statement is true for $k=3$ and $g\le 8$ (see Table \ref{tab:g3-values}), for other values we can use Lemma~\ref{lemma:incr-decr-s-bounds} and it is sufficient to prove $|G| \ge b_2(k, g, (k - 2)g)$ which is equivalent to $|G| > b_1(k, g, (k - 2)g)$.    
    This is clearly true for $s = 0$. We will proceed by induction as follows. Let $G$ be a~nontrivial $(k, g, s)$-multipole of order $n$ with $0 < s \le (k - 2)g$ and assume that Theorem \ref{thm:kgs-pole-lower} holds for all nontrivial $(k', g', s')$-multipoles of order $n'$ with $s' \le (k' - 2)g'$ where $(k', g', s', n')$ is lexicographically smaller than $(k, g, s, n)$. 
    By Lemma \ref{lemma:not-too-many-semiedges} we also assume that $G$ contains no vertex with inner degree $1$.
    
    \medskip
    \noindent\textbf{Case 1: $G$ contains a vertex $v$ with inner degree $2$}

   If $k \ge 4$, then let $G'=G-v$. Clearly, $G'$ is a $(k, g, s-(k-4))$-multipole of order $n-1$. Graph $G'$ cannot be trivial, otherwise either $G$ would be trivial, or it would have girth less than $g$ ($\ge 5$). But then 
   $G'$ fulfils the theorem condition and as $(k, g, s-(k-4), n-1)$ is lexicographically smaller than $(k, g, s, n)$ we have $|G'| \ge b_2(k, g, s-(k-4))$.
   By Lemma~\ref{lemma:incr-decr-s-bounds}, $|G'| \ge b_2(k, g, (k - 2)g))$ and then $|G| \ge b_2(k, g, (k - 2)g)+1$. Thus, $k=3$.

   For $k=3$ we already resolved  the cases when $g \le 8$. Thus, we may assume $g \ge 9$.
   We remove the vertex $v$ from $G$ and join the two links formerly incident with $v$ to produce one new link. The constructed multipole $G'$ is clearly a nontrivial $(3, g - 1, s - 1)$-multipole with $s - 1 \le (k - 2)(g - 1)$. By the induction hypothesis and Lemma \ref{lemma:incr-decr-s-bounds} we have $|G| \ge b_2(3, g - 1, s - 1) \ge b_2(3, g - 1, g - 1)$. So it is sufficient to show that
   $$b_2(3, g - 1, g - 1) > b_1(3, g, g)$$
   We manually checked that this is true for all $g < 14$. For $g \ge 14$, by Proposition \ref{prop:simple-bounds} it is sufficient to show
   \begin{equation}\label{eq:two-moores}
    M(3, g - 1) - \frac{(g - 1)^2}{2} > \frac{g^2}{2}, \qquad \text{or equivalently,} \qquad M(3, g - 1) > g^2 - g + \frac12.
   \end{equation}
   For $g = 2d + 1$, Inequality (\ref{eq:two-moores}) is equivalent to 
   \begin{equation}\label{eq:case1-odd}
   M(k, 2d) = 2\cdot 2^d - 2 > 4d^2 + 2d + \frac12,
   \end{equation}
   which is true for $g \ge 7$.
  
   For $g = 2d + 2$, Inequality (\ref{eq:two-moores}) is equivalent to
   \begin{equation}
       M(3, 2d + 1) = 3\cdot 2^d - 2  > 4d^2 + 6d + \frac{11}{2},
   \end{equation}
   which holds for $d \ge 6$, thus $g \ge 14$.

   \pagebreak
   
   \noindent\textbf{Case 2: Each vertex of $G$ has inner degree at least $3$}

    In this case, we have that $|G| \ge M(3, g)$. Also note, that in this case, we have $k \ge 4$. If $g \ge 9$, we show that
    \begin{equation}
        \label{eq:upper-proof-2e}
        M(3, g) > \frac{g^2}{2}.
    \end{equation}
    That is for $g = 2d + 1$ equivalent to
    $$3\cdot2^d - 2 > 2d^2 + 2d + \frac12,$$
    which is true for each $d \ge 4$, that is for each $g \ge 9$. If $g = 2d$, Inequality (\ref{eq:upper-proof-2e}) is equivalent to
    $$2^{d+1} - 2 > 2d^2,$$
    which is true for each $d \ge 5$, that is $g \ge 10$. Since $k \ge 4$ and $g \ge 9 \ge 7$, we have by Inequality (\ref{eq:upper-proof-2e}) and Proposition \ref{prop:simple-bounds} that $|G| > g^2/2 > b_1(k, g, s)$, so $|G| \ge b_2(k, g, s)$.

    For $5 \le g \le 8$ and $3 \le k \le 6$, we manually checked that $M(3, g) > b_1(k, g, s)$.

    Now we assume that $5 \le g \le 8$ and $k \ge 7$. We claim that $G$ has at least one vertex incident with no semiedge. Suppose the contrary. Then each vertex of $G$ is incident with a~semiedge. By removing one semiedge from each vertex of $G$, we obtain a $(k - 1, g, s - |G|)$-multipole $G'$ that is clearly nontrivial and $s - |G| \le (k - 2)g - g \le (k - 3)g$. Thus, by induction hypothesis $|G| = |G'| \ge \frac{1}{2}M(k - 1, g)$. However, one can easily prove by induction that $\frac{1}{2}M(k - 1, g) > (k - 2)g$, so $s \ge |G| > (k - 2)g$, which is a contradiction.

    Let $v$ be a vertex of $G$ incident with no semiedge. If $g = 2d + 1$, then applying Proposition \ref{prop:moore-bound-multipoles} on $v$ we get
    $$|G| \ge \frac{k(k - 1)^d - 2 - (k - 2)\left((k - 1)^{d - 1} - 1\right)}{k - 2}.$$
    If $g = 2d + 2$, we apply Proposition \ref{prop:moore-bound-multipoles} on some edge $vu$ incident with $v$. We obtain the smallest lower bound if $u$ has $k - 3$ semiedges (distance $0$) and all the remaining at most $s - k + 3 = (k - 2)g - k + 3$ semiedges are at distance $1$ from $uv$. So
    $$|G| \ge \frac{2(k - 1)^{d + 1} + s - (k - 3)(k - 1)^d - ((k - 2)g - k + 3)(k - 1)^{d - 1}}{k - 2}.$$
    Thus for specific values of $g$, we have
    \begin{itemize}
        \item If $g = 5$, then $|G| \ge k^2 - 5k + 11$, which is greater than $g^2/2$ for each $k \ge 6$.
        \item If $g = 6$, then $|G| \ge k^2 - 4k + 11$, which is greater than $g^2/2$ for each $k \ge 6$.
        \item If $g = 7$, then $|G| \ge k^3 - 8k^2 + 15k + 1$, which is greater than $g^2/2$ for each $k \ge 7$.
        \item If $g = 8$, then $|G| \ge k^3 - 7k^2 + 13k + 3$, which is greater than $g^2/2$ for each $k \ge 6$.
    \end{itemize}

\end{proof}

\section{Results on the connectivity of cages}
\label{sec:results}

\begin{theorem}
    \label{thm:kg-graphs-cc}
    If $G$ is a $(k, g)$-graph, where $k \ge 3$ and $g \ge 5$, of order
    $$|G| < 2b_2(k, g, (k - 2)g)$$
    then $G$ is cyclically $(k - 2)g$-edge-connected and each cycle separating $(k - 2)g$-edge-cut separates $G$ in two components one of which is a cycle of length $g$.
\end{theorem}

\begin{proof}
    For $(k, g)=(3, 6)$ we have $2b_2(3,6,6)=16$, thus the preconditions of the theorem hold only for the Heawood graph on 14 vertices. It is easy to check that the theorem holds for the Heawood graph. Thus assume $(k,g) \ne (3,6)$.   

    Let $S$ be a smallest cycle-separating edge-cut $S$ of size  $s \le (k - 2)g$. Since $S$ is a smallest cut, $G - S$ contains exactly two components $G_1$ and $G_2$, both being $(k, g, s)$-multipoles. If both $G_1$ and $G_2$ were nontrivial, then by Theorem \ref{thm:kgs-pole-lower} and Lemma \ref{lemma:incr-decr-s-bounds} we would have $|G| = |G_1| + |G_2| \ge 2b_2(k,g, s) \ge 2b_2(k, g, (k - 2)g)$, which is a contradiction. Thus one of $G_1$ and $G_2$ is trivial. Since both are cyclic, the only possibility is that $s = (k - 2)g$ and the trivial $(k, g, s)$-multipole $G_i$ is a cycle of length $g$, which proves the theorem.
\end{proof}

Now we list several pairs $(k, g)$ satisfying Conjectures \ref{conj:cc} and \ref{conj:g-cut}. This follows from the fact, that there is known a $(k, g)$-graph (not necessarily a cage) of order smaller than $2b_2(k, g, (k - 2)g)$.

\begin{theorem}
\label{thm:small-cages}
Let $k$ and $g$ be integers satisfying one of the following conditions
\begin{itemize}
    \item[(i)] there is a $(k, g)$-Moore graph,
    \item[(ii)] $k = 3$ and $g \le 15$,
    \item[(iii)] $k \in \{4, 5\}$ and $g \le 12$,
    \item[(iv)] $k \in \{6, 9\}$ and $g = 7$,
\end{itemize}
Then each $(k, g)$-cage $G$ is cyclically $(k - 2)g$-edge-connected and each cycle-separating $(k - 2)g$-edge-cut in $G$ separates $G$ into two components one of which is a cycle.
\end{theorem}

\begin{proof}
For $g \in \{3,4\}$ the cages are complete and complete bipartite graphs and it is easy to prove that these graphs fulfil the theorem. Otherwise,  
by Theorem \ref{thm:kg-graphs-cc}, it is sufficient to show that $n(k, g) < 2b_2(k, g, (k - 2)g)$. This is trivial for Case (i). For cases (ii) -- (iv) it follows from individual upper bounds on $n(k, g)$ taken from \cite{Exoo}. 
\end{proof}

According to Theorem \ref{thm:small-cages}, the cyclic edge-connectivity of a hypothetical $(57, 5)$-Moore graph is $55 \cdot 5 = 275$ and each cycle-separating $275$-edge-cut separates a $5$-cycle.

Our results can also be used to improve lower bounds on the size of a cyclic part that was provided by Nedela and Škoviera \cite{Nedela95}.
According to them, a \emph{cyclic part} is an induced subgraph of a cubic graph that can be separated by a cycle-separating edge-cut of the minimal cardinality. Clearly, a cyclic part of a cubic graph with cyclic edge-connectivity $c$ is a $(3, c, c)$-multipole. Therefore, by Theorem \ref{thm:kgs-pole-lower}, we have the following bound.

\begin{corollary}
Let $G$ be a cubic graph with cyclic edge-connectivity $c$ and let $P$ be a nontrivial cyclic part of $G$. Then
$$|P| \ge b_2(3, c, c) \ge \frac{1}{2} M(3, g).$$
\end{corollary}

Nedela and Škoviera provided the bound $|P| \ge 2c - 4$, or $|P| \ge 2c - 3$ if $c \ne 6$. Although this bound is tight for $c \in \{5, 6\}$, it is only linear. We have provided an exponential bound, which is a significant improvement for larger values of $c$.

Concluding our paper, we briefly discuss when a $(k, g,s)$-multipole should be considered nontrivial if $s > (k - 2)g$.
For instance, Figure \ref{fig:17-poles} illustrates $(3, 11, 17)$-multipoles of orders $17$, $19$, $21$ and $23$.
At first glance, there is no reason to exclude all of them from the definition of $n(k, g, s)$ and excluding only the $17$-cycle does not seem reasonable.
However, Propositions \ref{prop:quadratic-odd} and \ref{prop:quadratic-even} show that, for sufficiently small $s$, there is a~clear natural gap that separates nontrivial $(k, g, s)$-multipoles from trivial ones, since there are no $(k, g, s)$-multipoles of order between $b_1(k, g, s)$ and $b_2(k, g, s)$. According to this, all $(3, 11, 17)$-multipoles from Figure \ref{fig:17-poles} are trivial, since $b_1(3, 11, 17) = 43$ and $b_2(3, 11, 17) = 51$, so there are not $(3, 11, 17)$-multipoles whose order is strictly between $43$ and $51$. Furthermore, we checked by a computer program that there is no $(3, 11, 17)$-multipole of order at least $25$ and at most $43$. This suggests that the bounds on the size of a $(k, g, s)$-multipole from Propositions \ref{prop:quadratic-odd} and \ref{prop:quadratic-even} are not optimal. Also, the smallest $(3, 11, 17)$-multipole of order larger than $23$ has at least $51$ vertices.

\begin{figure}[h]
\begin{subfigure}[b]{0.24\linewidth}
\includegraphics[]{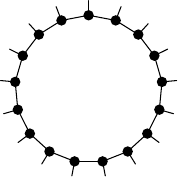}
\centering
\caption{order $17$}
\end{subfigure}
\begin{subfigure}[b]{0.24\linewidth}
\includegraphics[]{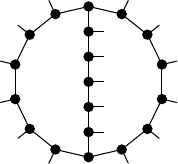}
\centering
\caption{order $19$}
\end{subfigure}
\begin{subfigure}[b]{0.24\linewidth}
\includegraphics[]{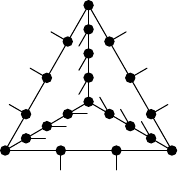}
\centering
\caption{order $21$}
\end{subfigure}
\begin{subfigure}[b]{0.24\linewidth}
\includegraphics[]{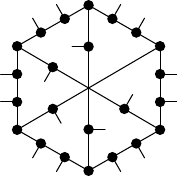}
\centering
\caption{order $23$}
\end{subfigure}
\caption{Examples of small cyclic $(3, 11, 17)$-multipoles}
\label{fig:17-poles}
\end{figure}

We finish the paper by providing a straightforward upper bound on $n(k, g, s)$.
Naturally, small nontrivial $(k, g, s)$-multipoles can be constructed from cages by removing some vertices or severing some edges, except for some cases of small $k$ and $g$, where such operations leave a~too small, and thus trivial, multipole.

\begin{proposition}\label{prop:mutlipoles-upper-bound}
Let $k \ge 3$, $g \ge 3$ and $0 < s < (k - 2)g$ be integers such that (i) $k$~is odd or $s$ is even, (ii) $k + g \ge 9$, and (iii) if $s$ is odd then $s \ge k - 2$. Then
$$n(k, g, s) \le n(k, g) - \left\lfloor \frac{s - 2}{k - 2} \right\rfloor + \left( k\left\lfloor \frac{s - 2}{k - 2} \right\rfloor + s \right) \bmod 2$$
and $n(k, g, (k - 2)g) \le n(k, g) - g$.
\end{proposition}

\begin{proof}
Let $G$ be a $(k, g)$-cage. Let $i = \lfloor (s - 2)/(k - 2)\rfloor$ and $j = s - i(k - 2) - 2$. Observe that $j \equiv ki + s \pmod2$, so for $2 \le s < (k - 2)g$ we want to prove that $n(k, g, s) \le n(k, g) - i + j \bmod 2$. Since $s \le (k - 2)g$, we have $i \le g$. Note that if $k$ is even, then $s = i(k - 2) + 2 + j$ has to be even due to (i), so $j$ is also even. Conversely, if $j$ is odd, then $k$ is also odd. Starting with $G$, we construct a nontrivial $(k, g, s)$-multipole $H$ as follows.

\paragraph{Construction 1.} If $s = (k - 2)g$, we remove a cycle of length $g$ from $G$ obtaining a~$(k, g, s)$-multipole $H$ of order $n(k, g) - g$.

\paragraph{Construction 2.} If $j$ is even, we remove a path consisting of $i$ vertices from $G$ and sever $j/2$ edges in $G$ obtaining a $(k, g, s)$-multipole $H$ of order $n(k, g) - i$.

\paragraph{Construction 3.}  If $j$ is odd, then $k$ is also odd and $i \ge 1$ due to (iii). We remove a path consisting of $i - 1$ vertices and sever $(k + 2 + j)/2$ edges in $G$ which yields a~$(k, g, s)$-multipole $H$ of order $n(k, g) - i + 1$.

\bigskip

We have constructed a $(k, g, s)$-multipole $H$ of the desired order which is at least $n(k, g) - g \ge M(k, g) - g$ in all cases. It is a matter of a simple induction to check that $M(k, g) - g > g$ if $k + g \ge 9$, so that means that $H$ is a nontrivial $(k, g, s)$-multipole.
\end{proof}

This construction is not optimal in general. For instance, we have constructed a~nontrivial $(4, 5, 10)$-multipole of order $14$ by removing a $5$-cycle from the Robertson graph of order $19$, which is a $(4, 5)$-cage. However, the smallest nontrivial $(4, 5, 10)$-multipole has order $10$ -- the Petersen graph with one semiedge attached to each vertex.
On the other hand, we show that this construction yields smallest nontrivial $(k, g, s)$-multipoles for $g = 3$ and $g = 4$.

\begin{proposition}\label{prop:g34}
Let $k \ge 3$, $g \in \{3, 4\}$ and $s < (k - 2)g$ be integers such that $k$ is odd or $s$ is even, and $k + g \ge 9$. Then
$$n(k, g, s) = n(k, g) - \left\lfloor \frac{s - 2}{k - 2} \right\rfloor + \left( k\left\lfloor \frac{s - 2}{k - 2} \right\rfloor + s \right) \bmod 2$$
and $n(k, g, (k - 2)g) = n(k, g) - g$.
\end{proposition}

\begin{proof}
    According to Proposition~\ref{prop:quadratic-odd}, the order $n$ of a $(k, 3, s)$-multipole satisfies
    \begin{equation}
    n^2 - (k + 1)n + s \ge 0. \label{eq:g3}
    \end{equation}
    If $G$ is a $(k, 4, s)$-multipole of order $G$, by Mantel's theorem $G$ has at most $n^2/4$ links. On the other, hand, the number of links of $G$ is $(kn - s)/2$, thus we have
     \begin{align}
         \frac{kn - s}{2} &\le \frac{n^2}{4},\nonumber\\
         n^2 - 2kn + 2s &\ge 0. \label{eq:g4} 
     \end{align}
    A straightforward analysis of Equations~(\ref{eq:g3})~and~(\ref{eq:g4}), with eventually adding $1$ when $k\lfloor (s - 2)/(k - 2) \rfloor + s$ is odd, yields that the upper bound from Proposition~\ref{prop:mutlipoles-upper-bound} is reached for $g \in \{3, 4\}$. 
\end{proof}

\section*{Acknowledgements}

The authors were partially supported by VEGA 1/0727/22 and APVV-23-0076.
Also, the authors would like to thank Gunnar Brinkmann for providing the computer program \texttt{multigraph} used for finding $n(k, g, s)$ for small parameters.

\bibliographystyle{abbrvnat}
\newcommand*{\doi}[1]{doi: \href{https://doi.org/#1}{#1}}
\bibliography{bibliography}

\end{document}